\documentclass[11pt]{article}

\usepackage{graphicx}
    \graphicspath{ {figures/} }
    \usepackage[margin=0.1cm]{caption}
    \usepackage{float}
\usepackage{amsmath}
\usepackage{amsthm}
\usepackage{amssymb}
\usepackage[mathscr]{eucal}
\usepackage[all]{xy}
\usepackage{epsfig}
\usepackage{epstopdf}
\usepackage{amscd}
\usepackage[normalem]{ulem}
\usepackage{enumitem}
    \setlist[enumerate,1]{label=\textnormal{(\alph*)}}
\usepackage{thmtools, thm-restate}
\usepackage{array,ragged2e}

\usepackage[left=1.25in,right=1.25in,top=1.25in,bottom=1.25in,
            ]{geometry}
\usepackage{hyperref}
\usepackage{color}

\theoremstyle{plain}
\newtheorem{theorem}{Theorem}

\newtheorem{lemma}[theorem]{Lemma}

\theoremstyle{remark}

\theoremstyle{definition}
\newtheorem{question}{Question}
\newtheorem*{question*}{Question}

\newcommand{\mcc}{m_{c}}
\newcommand{\tcc}{t_{c}}

\makeatletter
\newcommand\xleftrightarrow[2][]{%
  \ext@arrow 9999{\longleftrightarrowfill@}{#1}{#2}}
\newcommand\longleftrightarrowfill@{%
  \arrowfill@\leftarrow\relbar\rightarrow}
\makeatother

\usepackage{authblk}
\title{Knot Mosaics with Corner Connection Tiles }
\author[]{Aaron Heap, Una Donovan, Riley Grossman, Nickolas Laine, Connor McDermott, Marcus Paone, and Drew Southcott}
\date{}                     
\begin{document}

\maketitle
\begin{abstract}
A knot mosaic is a representation of a knot or link on a square grid using a collection of tiles that are either blank or contain a portion of the knot diagram. Traditionally, a piece of the knot on one tile connects to a piece of the knot on an adjacent tile at a connection point that is located at the midpoint of a tile edge. We introduce a new set of tiles in which the connection points are located at corners of the tile. By doing this, we can create more efficient knot mosaics for knots with small crossing number. In particular, when using these corner connection tiles, it is possible to create knot mosaic diagrams for all knots with crossing number 8 or less on a mosaic that is no larger and uses fewer non-blank tiles than is possible using the traditional tiles.
\end{abstract}


\section{Introduction}\label{section:introduction}

Lomonaco and Kauffman introduced the study of knot mosaics as a system that provides a blueprint for constructing a physical quantum knot system \cite{Lom-Kauff}. They conjectured that their formal knot mosaic system completely captured the entire structure of tame knot theory. This was proven to be true by Kuriya and Shehab, when they proved that knot mosaic theory was equivalent to tame knot theory \cite{Kuriya}.

A \emph{knot mosaic} is a two-dimensional square array representation of a knot (or link) made up of a finite set of tiles. Traditionally, the tiles used for constructing these mosaics have been those shown in Figure \ref{fig:tiles-traditional}. The first tile is a blank tile, and we refer to the other tiles collectively as non-blank tiles. Each non-blank tile displays a piece of the knot with an endpoint located at the midpoint of a tile edge. This endpoint is called a \emph{connection point} of the tile. In order to be \emph{suitably connected}, each connection point on a tile must meet a connection point on an adjacent tile.

\begin{figure}[ht]
  \centering
  \includegraphics{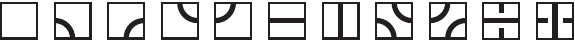}\\
  \caption{Collection of tiles used to construct knot mosaics.}
  \label{fig:tiles-traditional}
\end{figure}

An \emph{$n\times n$ knot mosaic}, or \emph{$n$-mosaic}, is a knot mosaic with $n$ rows and $n$ columns. Some examples of knot mosaics are shown in Figure \ref{fig:examples-traditional}, where we see a 4-mosaic, three 5-mosaics, and a 6-mosaic. Much work has been done to find efficient ways of depicting a particular knot $K$ as a knot mosaic. What is the smallest $n$ such that $K$ can fit on an $n$-mosaic? This is the \emph{mosaic number} of $K$, denoted $m(K)$. What is the fewest number of non-blank tiles needed to depict $K$? This is the \emph{tile number} of $K$, denoted $t(K)$. These efficiency questions have been answered for every prime knot with crossing number 10 or less, and every knot with mosaic number 6 or less has been determined. The reader can find details in \cite{Heap4,Heap2,Heap3,Kuriya,Lee2}. Each of the mosaics given in Figure \ref{fig:examples-traditional} are the most space-efficient knot mosaics possible for the given knots, which means both the mosaic number and tile number are realized. A knot mosaic is said to be \emph{space-efficient} if the number of non-blank tiles is as small as possible and cannot be made to fit on a smaller mosaic. However, it may be possible to decrease the number of non-blank tiles by increasing the size of the mosaic.

\begin{figure}[ht]
  \centering
  \includegraphics{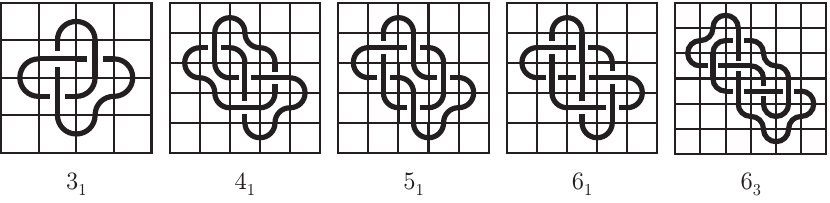}\\
  \caption{Examples of knot mosaics using traditional tiles.}
  \label{fig:examples-traditional}
\end{figure}

One could consider creating mosaics using other types of tiles. For example, there has been some exploration of mosaics made using hexagonal tiles, as seen in \cite{Hex-tiles-BushComminsGomezMcLoud-Mann} and \cite{Hex-tiles-Howards}. In this paper, we introduce a set of square tiles, labeled $T_0$ to $T_{10}$ and shown in Figure \ref{fig:tiles-new}, that can be used to construct knot mosaics but are different than the original set of square tiles given in Figure \ref{fig:tiles-traditional}. Instead of having the connection point located at the midpoint of a tile edge, these new tiles have connection points located at the corners of the tiles. We refer to these tiles as \emph{corner connection tiles}. The first tile $T_0$ is blank, the next four tiles $T_1$ to $T_4$ are \emph{single-arc tiles}, $T_5$ and $T_6$ are \emph{diagonal line segment tiles}, $T_7$ and $T_8$ are \emph{double-arc tiles}, and $T_9$ and $T_{10}$ are \emph{crossing tiles}.

\begin{figure}[ht]
  \centering
  \includegraphics{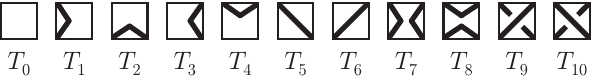}\\
  \caption{Tiles with corner connection points.}
  \label{fig:tiles-new}
\end{figure}

In Figure \ref{fig:examples-new} we see examples of knot mosaics created using these new tiles. Notice that in each case, the knots depicted can be created more efficiently with these new tiles than they could with the traditional tiles. For example, the unknot has mosaic number 2 and tile number 4 when using traditional tiles, but with the new tiles we can create the unknot on a 2-mosaic using only two non-blank tiles. Similarly, the trefoil knot $3_1$ has mosaic number 4 and tile number 12, but with the corner connection tiles the trefoil can be constructed on a 3-mosaic using only 8 non-blank tiles. The efficiency advantage comes from the fact that, unlike with traditional tiles, using corner connection tiles allows for crossing tiles to be placed in the outermost rows and columns of the mosaic.

\begin{figure}[ht]
  \centering
  \includegraphics[scale=.8]{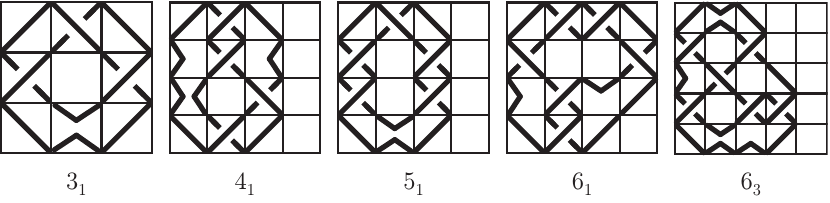}\\
  \caption{Examples of knot mosaics using the new tiles.}
  \label{fig:examples-new}
\end{figure}

We note that the knot $6_1$, the fourth mosaic in Figure \ref{fig:examples-traditional}, is an example of a knot where the crossing number cannot be realized on a minimal knot mosaic when using the traditional tiles. In fact, Ludwig, Evans, and Paat \cite{Ludwig} found an infinite family of knots exhibiting this phenomenon. When using the corner connection tiles, we will see below that this issue does not occur for any prime knot with crossing number 8 or less, and it is not yet known if there are any knots whose crossing number cannot be realized on a minimal mosaic.

If we let $K$ be a knot, we define the \emph{corner connection mosaic number} of $K$, denoted $\mcc(K)$, to be the smallest positive integer $n$ such that $K$ can fit on an $n$-mosaic using corner connection tiles. We similarly define the \emph{corner connection tile number} of $K$, denoted $\tcc(K)$, to be the fewest number of non-blank tiles needed to depict $K$ on a knot mosaic using corner connection tiles. In Section \ref{section:layouts}, we determine the corner connection mosaic number for all prime knots with crossing number 8 or less. We also determine the corner connection tile number for many of these.

In each of the examples given in Figures \ref{fig:examples-traditional} and \ref{fig:examples-new}, we see that $\mcc(K) \leq m(K)$ and $\tcc(K) \leq t(K)$. A natural question then arises. Is it always true that the traditional mosaic number and tile number will be greater than or equal to the corner connection mosaic number and corner connection tile number? In Section \ref{section:max-crossing}, we will see that the answer is no. Although more efficiency is obtained with these new tiles while the mosaic size is small, this is not always the case. As we will see, constraining the number of crossings and examining where they can be located will be a major tool for showing this. In Section \ref{section:rectangular}, we briefly explore the concept of a rectangular mosaic and provide a maximum for the number of crossing tiles allowed, and in Section \ref{section:questions} we provide a collection of open questions that readers may wish to explore.

\section{Knot Mosaics for Prime Knots}\label{section:layouts}

Our initial goal is to find space-efficient knot mosaic layouts for prime knots using the corner connection tiles. In order for the knot diagram to be reduced and the knot mosaic to be space-efficient, we must avoid unnecessary crossings caused by kinks in the knot diagram and we attempt to minimize the number of non-blank tiles used in the mosaic. In other words, we try to avoid unnecessary meandering through the mosaic. We will also avoid multi-component links in this section, as our focus will be strictly limited to prime knots.

In many of the figures that follow, we will make use of nondeterministic tiles such as a crossing tile without an indication of over- or under-crossing or a four-connection-point tile which could be a crossing or double-arc tile $T_7$, $T_8$, $T_9$, or $T_{10}$. The nondeterministic tiles with dashed segments could be blank tiles or non-blank tiles using any of the indicated connection points, but they could also indicate that a tile location is optional. Examples are shown in Figure \ref{fig:tiles-nondeterministic}. The third tile indicates that $T_0$, $T_1$, $T_4$, or $T_6$ can be used in its place or no tile at all. In particular, this tile indicates that the lower-right connection point cannot be used. The fourth tile indicates that the upper-left connection point must be used, which includes $T_1$, $T_4$, $T_5$, $T_7$, $T_8$, $T_9$, or $T_{10}$. The fifth tile indicates that this location could be any tile or no tile at all.

\begin{figure}[ht]
  \centering
  \includegraphics{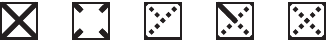}\\
  \caption{Examples of nondeterministic tiles.}
  \label{fig:tiles-nondeterministic}
\end{figure}

In traditional knot theory, we use planar isotopy moves to deform knot diagrams into different but equivalent knot diagrams. Analogously, we can move parts of the knot within a mosaic using \emph{mosaic planar isotopy moves}, replacing one set of tiles with another, to obtain another knot mosaic diagram that does not change the knot type depicted in the mosaic. A complete list of all of the moves when using the traditional mosaic tiles is given in \cite{Kuriya} and \cite{Lom-Kauff}. We provide a few examples of these mosaic planar isotopy moves when using corner connection tiles in Figure \ref{fig:isotopies}.

\begin{figure}[ht]
  \centering
  \includegraphics[width=\linewidth]{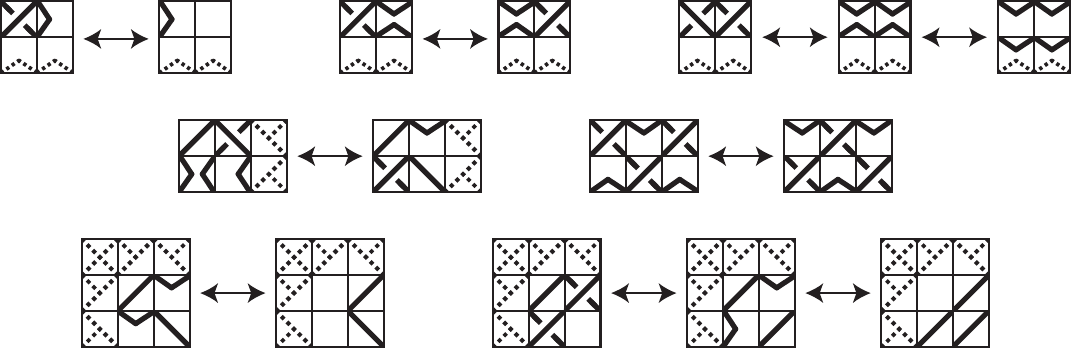}\\
  \caption{Examples of mosaic planar isotopy moves using corner connection tiles.}
  \label{fig:isotopies}
\end{figure}

As we determine what mosaic layouts are possible using the corner connection tiles, we point out that the corner tiles of the knot mosaic cannot be a crossing or double-arc tile. Because of this, a 2-mosaic cannot have a crossing, and the unknot is the only knot with mosaic number 2, using either traditional tiles or the corner connection tiles. In fact, the only possibilities for the corner position of a space-efficient mosaic are a blank tile or a tile with a diagonal line segment, and we generalize this further in the following lemma.

\begin{lemma}\label{lemma:start-end-tiles}
For a space-efficient knot mosaic created from corner connection tiles, we may assume the first non-blank tile in the first non-blank row is a diagonal line segment tile $T_6$. We may also assume that the last non-blank tile in this row is $T_5$ and that there is at least one non-blank tile between the $T_5$ and $T_6$ tiles. Similarly, we may assume the last non-blank row and the first and last non-blank columns begin and end with diagonal line segment tiles $T_5$ or $T_6$ with at least one tile between them.
\end{lemma}


\begin{proof}
  We can assume that the top row of the mosaic is not blank by shifting all tiles up, if necessary. Consider the first non-blank tile (from left to right) in this first row. It clearly cannot have four connection points, otherwise the connection point in the top, left corner of this tile would not be suitably connected. So this first tile must have two connection points, meaning it is a single-arc tile ($T_2$ or $T_3$) or a diagonal line segment ($T_6$). If the first non-blank tile is a single-arc tile $T_3$ with a connection point that touches the top of the mosaic, then this connection point must suitably-connect with the adjacent tile to the right. As we see in the first part of Figure \ref{fig:corner-tile-single-arc}, this will always lead to an unlinked component, a kink, or space-inefficiency. We get a similar result if it is a single-arc tile $T_2$ that is in the corner of the mosaic.

  \begin{figure}[ht]
  \centering
  \includegraphics[width=\linewidth]{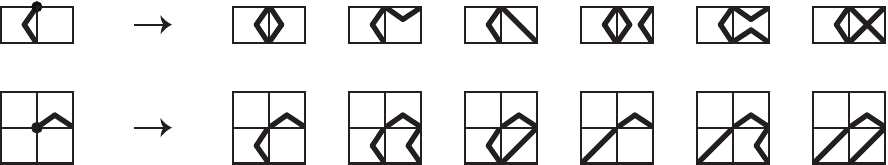}\\
  \caption{A single arc tile as the first non-blank tile of the first row results in an unknotted component, space-inefficiency, or a mosaic that can be modified via planar isotopy to remove it from the first row.}
  \label{fig:corner-tile-single-arc}
\end{figure}

  If the first non-blank tile is a single-arc tile $T_2$ that is not in the corner (i.e. has a blank tile in the preceding column), then the left connection point must suitably-connect with the tile directly below it or with the tile diagonally below to the left. If it is the tile directly below, the result will be symmetrically equivalent to those in the first part of Figure \ref{fig:corner-tile-single-arc}. If it is connected to the tile below and to the left, then the tile directly below either is blank or is non-blank with inefficiency or unlinked components, as shown in the second part of Figure \ref{fig:corner-tile-single-arc}. In the case that the tile below is blank, then the mosaic is not space-efficient or the single-arc tile in question can be flipped down, removing it from the first row. Therefore, we may assume that the first non-blank tile in the first row is a diagonal line segment $T_6$. By symmetry, we are able to conclude that the last non-blank tile of this row is a diagonal line segment $T_5$. If the $T_6$ tile is immediately followed by the $T_5$ tile, the result is either a link or the mosaic is not space-efficient, as seen in Figure \ref{fig:first-row-tiles}. By symmetry, we are able to reach an analogous conclusion for the first and last non-blank row or column.
\begin{figure}[ht]
  \centering
  \includegraphics[width=\linewidth]{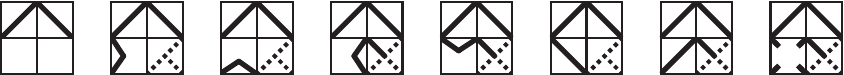}\\
  \caption{The mosaic is not space-efficient if there is no tile between the $T_6$ and $T_5$ tiles in the first row.}
  \label{fig:first-row-tiles}
\end{figure}
\end{proof}

\begin{lemma}\label{lemma:2by2}
  Any $2 \times 2$ sub-mosaic of a knot mosaic contains at most two tiles with four connection points $T_7$, $T_8$, $T_9$, or $T_{10}$. In particular, there can only be two crossing tiles in any $2 \times 2$ sub-mosaic.
\end{lemma}

\begin{proof}
In any $2\times 2$ sub-mosaic of a knot mosaic, all four tiles share a corner. This corner can only be a connection point for at most two of the tiles, otherwise the resulting mosaic would not depict a knot, as it would have three or more arcs coming together at a single point. Therefore, at least two of the tiles cannot have a connection point in this location.
\end{proof}


These lemmas make determining the layouts for a 3-mosaic simple. There are five possible positions to place crossing tiles. Furthermore, knowing that the unknot can be placed on a 2-mosaic, we can assume that the knot mosaic must have at least three crossing tiles. None of the corner tiles of the mosaic can be blank because this would prevent a crossing tile from being in either position adjacent to the corner, which would leave only three possible locations for the crossings in a $2\times 2$ sub-mosaic, and that would result in a violation of Lemma \ref{lemma:2by2}. So the corner tiles must be diagonal line segments. If there is a crossing in the center tile of the mosaic, then Lemma \ref{lemma:2by2} forces all three crossings to be in the middle column or middle row. If the center tile is not a crossing, then at least three of the four remaining boundary tiles must be crossings. Up to symmetry, the only possible layouts for a space-efficient 3-mosaic are those shown in Figure \ref{fig:3mosaic-layouts}. The first has 8 non-blank tiles, and the second has 9 non-blank tiles.

\begin{figure}[ht]
  \centering
  \includegraphics{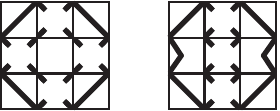}\\
  \caption{Possible space-efficient layouts for a 3-mosaic.}
  \label{fig:3mosaic-layouts}
\end{figure}

If $K$ is a nontrivial knot that can be created on one of these layouts, then we know the corner connection mosaic number of $K$ is 3. However, without doing a little more work (see Lemma \ref{lemma:minimal-nonblank-tiles}), one can only conclude that the corner connection tile number is no more than 8 or 9, respectively. When creating knot mosaics using traditional tiles, the tile number is not always realized on the smallest possible mosaic (see \cite{Heap2} for details). We have no reason to believe that such a phenomenon does not occur with the corner connection tiles, and we will see below in Theorems \ref{thm:4mosaic-knots} and \ref{thm:tile-number-not-minimal} that, in fact, it does occur.

Using traditional tiles, bounds for the tile number on an $n$-mosaic were determined in \cite{Heap1}, but we have yet to determine the bounds when using corner connection tiles. However, we can say a few things about the lower bound of the number of non-blank tiles. 

\begin{lemma}\label{lemma:minimal-nonblank-tiles}
  The minimum number of non-blank corner connection tiles necessary to create an $n$-mosaic, for $n \geq 3$, that cannot be made to fit on a smaller mosaic via mosaic planar isotopy moves is an increasing function of $n$. This minimum is $2n+2$ for $n=3$ and $n=4$. For $n \geq 5$, the minimum is greater than $2n+2$.
\end{lemma}

\begin{proof}
Suppose we have an $n$-mosaic, with $n \geq 3$, created using corner connection tiles, and assume that every column (or every row) is occupied so that the resulting mosaic cannot fit on a smaller mosaic. To achieve a nontrivial knot, we must have at least three rows occupied because of Lemma \ref{lemma:start-end-tiles}, and one could say that the minimum number of non-blank tiles is $2n+2$, created by using a non-blank tile in every location of the first and third row and only two non-blank tiles in the second row. See Figure \ref{fig:minimal-layouts}. This is fine for a 3-mosaic and 4-mosaic, where 8 and 10 are, respectively, the minimum number of non-blank tiles, but for $n \geq 5$, this layout can be made to fit on an $(n-1)$-mosaic via simple mosaic planar isotopy moves. In fact, we will see below in Theorem \ref{thm:5mosaic-knots} that the minimum number of non-blank tiles for a 5-mosaic (that cannot be made to fit on a smaller mosaic) is 13, not 12. Similarly, one can show that the minimum number of non-blank tiles on a 6-mosaic is at least 16, not 14. However, it is clear that the minimum number of non-blank tiles on an $n$-mosaic is an increasing function of $n$.
\end{proof}

\begin{figure}[ht]
  \centering
  \includegraphics[scale=.8]{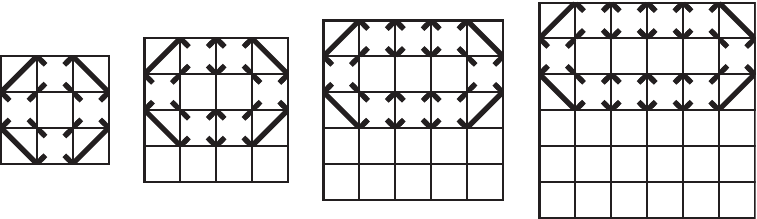}\\
  \caption{Possible minimal layouts for an $n$-mosaic.}
  \label{fig:minimal-layouts}
\end{figure}

\begin{theorem}\label{thm:3mosaic-knots}
  The only knot with corner connection mosaic number 3 is the trefoil knot, $\mcc(3_1)=3$, and the corner connection tile number of the trefoil knot is $\tcc(3_1)=8$.
\end{theorem}

\begin{figure}[ht]
  \centering
  \includegraphics{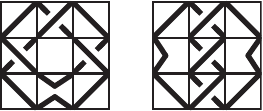}\\
  \caption{Possible nontrivial prime knots for a 3-mosaics.}
  \label{fig:3mosaic-knots}
\end{figure}

\begin{proof}
  We simply populate the layouts of Figure \ref{fig:3mosaic-layouts} with double-arc tiles and at least three crossing tiles. Note that using four crossing tiles in the first layout will result in a link. After eliminating links, kinks, and other inefficiencies, there are only two nontrivial possibilities, shown in Figure \ref{fig:3mosaic-knots}. Both of these depict the trefoil knot. We also see that there is no way to obtain the figure-8 knot $4_1$ (or any other knot with larger crossing number) on a 3-mosaic. We now know the corner connection mosaic number of the trefoil knot, and the corner connection tile number is no more than 8. Because of Theorem \ref{lemma:minimal-nonblank-tiles}, we know that we cannot use fewer non-blank tiles on a larger mosaic, and we are able to conclude that $\tcc(3_1)=8$.
\end{proof}

Now we examine 4-mosaics and determine the possible layouts for a space-efficient 4-mosaic depicting a prime knot, excluding those that fit on a smaller mosaic.

\begin{theorem}\label{thm:4mosaic-layouts} If we have a space-efficient 4-mosaic of a prime knot for which either every column or every row is occupied, then the only possible layouts (up to symmetry) are those given in Figure \ref{fig:4mosaic-layouts}. The number of non-blank tiles in each of these is 10, 11, 12, 12, 13, 13, and 14, respectively.
\begin{figure}[ht]
  \centering
  \includegraphics{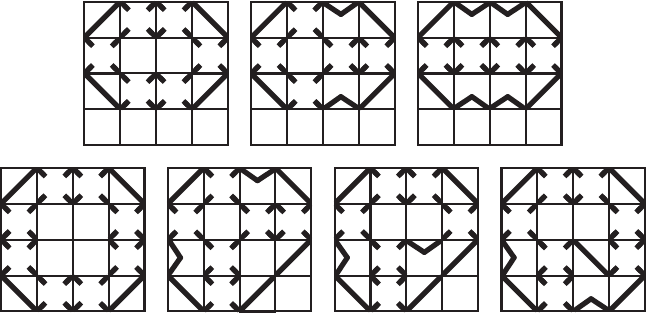}\\
  \caption{Possible space-efficient layouts for a 4-mosaic.}
  \label{fig:4mosaic-layouts}
\end{figure}
\end{theorem}

\begin{proof}
We may assume that the first column and first row of the mosaic contain non-blank tiles, and let us also assume that every column is occupied. We also need at least four crossing tiles. By Lemma \ref{lemma:2by2}, there can be at most two crossing tiles in the $2\times2$ interior.
At least two of the minimum of four crossings must be located in the boundary of the mosaic.

\begin{figure}[ht]
  \centering
  \includegraphics{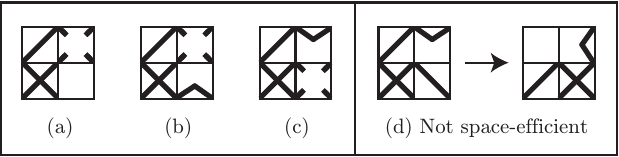}\\
  \caption{Possibilities for the upper-left $2 \times 2$ corner.}
  \label{fig:4mosaic-top-left2by2}
\end{figure}

Let $a_{ij}$ denote the tile position in row $i$, column $j$. Because of symmetry, we may assume that $a_{21}$ is one of the crossings. The diagonally adjacent $a_{12}$ must be a tile with four connection points or a single-arc tile $T_4$, as anything else would cause a link or inefficiency. If $a_{12}$ is a tile with four connection points, then $a_{22}$ must be blank or a single-arc $T_2$ (symmetrically equivalent to using a single-arc $T_3$). If $a_{12}$ is $T_4$, then $a_{22}$ must have four connection points or be a diagonal segment tile $T_5$. In this latter case, the result is not space-efficient, as we can move the crossing into the adjacent $a_{22}$ position and reduce the number of non-blank tiles. So the possibilities for the upper-left $2 \times 2$ corner are those shown in Figure \ref{fig:4mosaic-top-left2by2}, with the first three being the only space-efficient options.

Now examine the $a_{31}$ position. For option (a) in Figure \ref{fig:4mosaic-top-left2by2}, $a_{31}$ can be a diagonal segment tile $T_5$, a single-arc tile $T_1$, or a tile with four connection points. Once that choice is made, the rest of the lower-left $2\times 2$ corner is easy to complete. If we avoid obvious links and kinks, the resulting possibilities for the first two columns in this case are given in the first six partial mosaics in Figure \ref{fig:4mosaic-first2columns}. For option (b) of Figure \ref{fig:4mosaic-top-left2by2}, each resulting possibility is equivalent to one from the previous option. For option (c) of Figure \ref{fig:4mosaic-top-left2by2}, $a_{31}$ can be a diagonal segment tile $T_5$ or a single-arc tile $T_1$. Again, we can easily complete the rest of the lower-left $2\times 2$ corner, and the resulting viable possibilities are given in the remaining four diagrams of Figure \ref{fig:4mosaic-first2columns}.

\begin{figure}[ht]
  \centering
  \includegraphics[width=\linewidth]{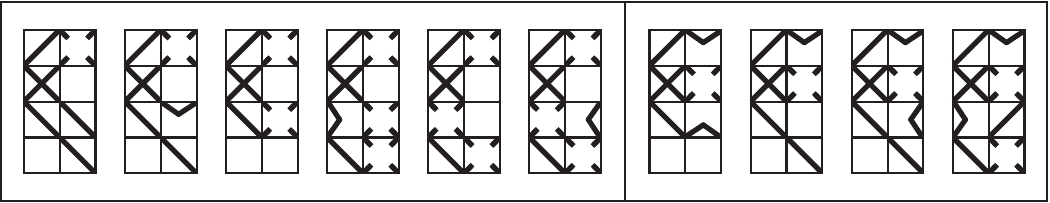}\\
  \caption{Possibilities for the first two columns.}
  \label{fig:4mosaic-first2columns}
\end{figure}

We note that the first option in Figure \ref{fig:4mosaic-first2columns} is not space-efficient, as the crossing can be moved into the $a_{32}$ position and the tile with four connection points can be moved to the $a_{22}$ position, resulting in an empty first column. For the remaining options, we complete the top two rows of the partial mosaics. Notice that in every case, the tile in the $a_{12}$ position has a connection point in the top right corner, which limits the options for the tile in the $a_{13}$ position. This, together with the fact that every column must be occupied by the knot diagram, leads to the possibilities given in Figure \ref{fig:4mosaic-top-right2by2} for the upper-right $2 \times 2$ corner.

\begin{figure}[ht]
  \centering
  \includegraphics{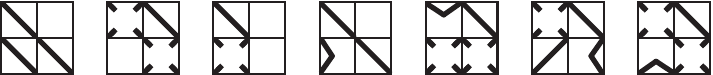}\\
  \caption{Possibilities for the upper-right $2 \times 2$ corner.}
  \label{fig:4mosaic-top-right2by2}
\end{figure}

We now adjoin all of these upper-right $2 \times 2$ corner possibilities to the options given in Figure \ref{fig:4mosaic-first2columns}, and we leave it to the reader to verify the following claims. The second option in Figure \ref{fig:4mosaic-first2columns} leads to only one space-efficient possibility, with the rest clearly being able to simplify to mosaics with fewer non-blank tiles. The third option leads to two partial mosaics, after eliminating space-inefficient options or those that can fit on a smaller mosaic. The fourth option also leads to two partial mosaics, again after eliminating inefficient layouts or those that are equivalent to those found in previous cases. We continue this process for each of the remaining options in Figure \ref{fig:4mosaic-first2columns}, eliminating any inefficiencies or symmetrical redundancy. The final set of possibilities are given in Figure \ref{fig:4mosaic-first2columns-2rows}.

\begin{figure}[ht]
  \centering
  \includegraphics[width=\linewidth]{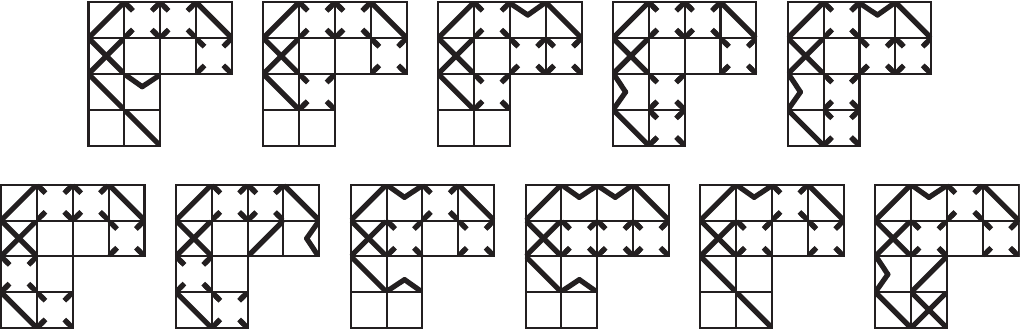}\\
  \caption{Possibilities for the first two columns and two rows of the 4-mosaic.}
  \label{fig:4mosaic-first2columns-2rows}
\end{figure}

Finally, we are able to complete the 4-mosaics. Taking the options given in Figure \ref{fig:4mosaic-first2columns-2rows}, we simply fill out the lower-right $2\times 2$ corner. For many, there is only one way to do this in order to avoid inefficiency in this corner. In some instances, even this leads to an overall inefficient mosaic. Omitting any redundancy, the final results are exactly those given above in Figure \ref{fig:4mosaic-layouts}.
\end{proof}

\begin{theorem}\label{thm:4mosaic-knots}
  The only prime knots with corner connection mosaic number 4 are $4_1$, $5_1$, $5_2$, $6_1$, $7_1$, and $7_2$. The corner connection tile numbers of the first five of these knots are $\tcc(5_1)=10$, $\tcc(4_1)=\tcc(5_2)=11$, $\tcc(7_1)=12$, and $\tcc(6_1)=13$. On a 4-mosaic, $7_2$ requires 14 non-blank tiles, and $\tcc(7_2)\leq 14$.
\end{theorem}

\begin{figure}[ht]
  \centering
  \includegraphics[width=\linewidth]{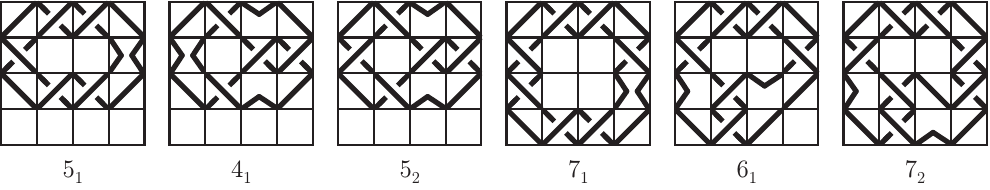}\\
  \caption{Prime knots with corner connection mosaic number 4.}
  \label{fig:4mosaic-knots}
\end{figure}

\begin{proof}
The corner connection mosaic number is easy to determine, as we have provided 4-mosaic diagrams in Figure \ref{fig:4mosaic-knots} for each of the knots. To obtain these and verify that they are the only resulting knots, we need to populate the layouts in Figure \ref{fig:4mosaic-layouts} with double-arc tiles and at least four crossing tiles. Note that if we place four or six crossing tiles on the first layout, the result is always a link and can be disregarded for our purposes. Placing five crossing tiles within the first layout in any alternating configuration results in the $5_1$ knot, while any non-alternating configuration reduces to the trefoil knot or the unknot. For the second layout, there are five possible locations to place the crossing tiles. When placing four alternating crossing tiles, two of them must be in the last two positions of the second row, otherwise the result is a link or kink. Regardless of where the remaining two crossings are placed, the result is the figure-8 knot $4_1$. If we place five alternating crossings in this layout, the result is the $5_2$ knot. Notice that if we place four crossing tiles on the third layout, the result is necessarily a link.

As we move to the next layout, we may now assume there are at least six crossing tiles. Similar to the first layout, an even number of crossing tiles on the fourth layout will always result in a link, so we only need to consider seven crossings. Any configuration of alternating crossings results in the knot $7_1$. Placing six crossings in the fifth layout results in a link, and placing six alternating crossings in the sixth layout produces the knot $6_1$. In the final layout, to avoid a link, both of the two locations in the last two entries of the second column must be a crossing. Populating the remaining five locations with any configuration of four alternating crossings produces the knot $6_1$ again, and populating them with five alternating crossings produces the knot $7_2$.

Similar to what we noted above for 3-mosaics, the knot mosaics in Figure \ref{fig:4mosaic-knots} technically only provide an upper bound for each corner connection tile number. We note that the smallest possible number of non-blank tiles on a space-efficient 5-mosaic is 13, as we will see below. We may conclude that $\tcc(5_1)=10$, $\tcc(4_1)=\tcc(5_2)=11$, $\tcc(7_1)=12$, and $\tcc(6_1)=13$.
\end{proof}

We now arrive at our first bit of intrigue. The knot $7_2$ needs 14 non-blank tiles to fit on a 4-mosaic, but it can be displayed on a 5-mosaic with fewer non-blank tiles. See Figure \ref{fig:5mosaic-7_2}. This mirrors what we see when using traditional tiles. However, using traditional tiles, the first occurrence of this phenomenon is the knot $9_{10}$ \cite{Heap1}.

\begin{figure}[ht]
  \centering
  \includegraphics{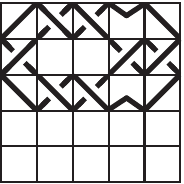}\\
  \caption{The knot $7_2$ on a 5-mosaic with only 13 non-blank tiles.}
  \label{fig:5mosaic-7_2}
\end{figure}

\begin{theorem}\label{thm:tile-number-not-minimal}
The corner connection tile number of $7_2$ is $\tcc(7_2)=13$.
\end{theorem}

\begin{proof}
To see that $\tcc(7_2)=13$, we again note that the smallest possible number of non-blank tiles on a space-efficient 5-mosaic is 13, and in Figure \ref{fig:5mosaic-7_2} we provide a 5-mosaic of $7_2$ with this minimum realized. We also note that the smallest possible number of non-blank tiles on a space-efficient 6-mosaic is 15.
\end{proof}

\begin{figure}[ht]
  \centering
  \includegraphics[width=\linewidth]{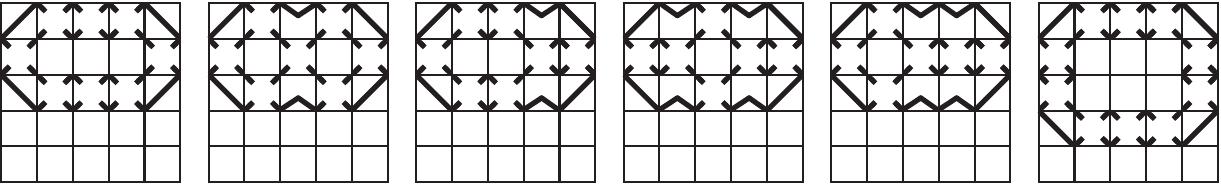}\\
  \caption{Layouts for a 5-mosaic with 14 or less non-blank tiles and at least 6 crossings locations.}
  \label{fig:5mosaic-layouts-14orless}
\end{figure}

The number of possible space-efficient layouts for a 5-mosaic is large, so we do not include them here. However, the simplest layouts with 14 or less non-blank tiles and with at least 6 crossings are those given in Figure \ref{fig:5mosaic-layouts-14orless}. The first of these layouts (with 12 non-blank tiles) is not actually space-efficient, as it can be made to fit on a 4-mosaic, equivalent to the fourth layout of Figure \ref{fig:4mosaic-layouts}.

%

While we do not determine every prime knot with corner connection mosaic number 5, we verify that every prime knot with crossing number 8 or less not listed previously in Theorems \ref{thm:3mosaic-knots} and \ref{thm:4mosaic-knots} has corner connection mosaic number 5. To see this, we provide a 5-mosaic for each of these in Figure \ref{fig:5mosaic-knots}. We have not verified that the minimum number of non-blank tiles is used in the provided mosaics for all of these, so we cannot definitively conclude the corner connection tile number for most of them. Perhaps there is a more space-efficient layout for some of these, possibly obtained by adding extra crossings. However, for every knot with crossing number 8 or less, the corner connection tile number is less than or equal to 20.

\begin{figure}[ht]
  \centering
  \includegraphics[width=.85\linewidth]{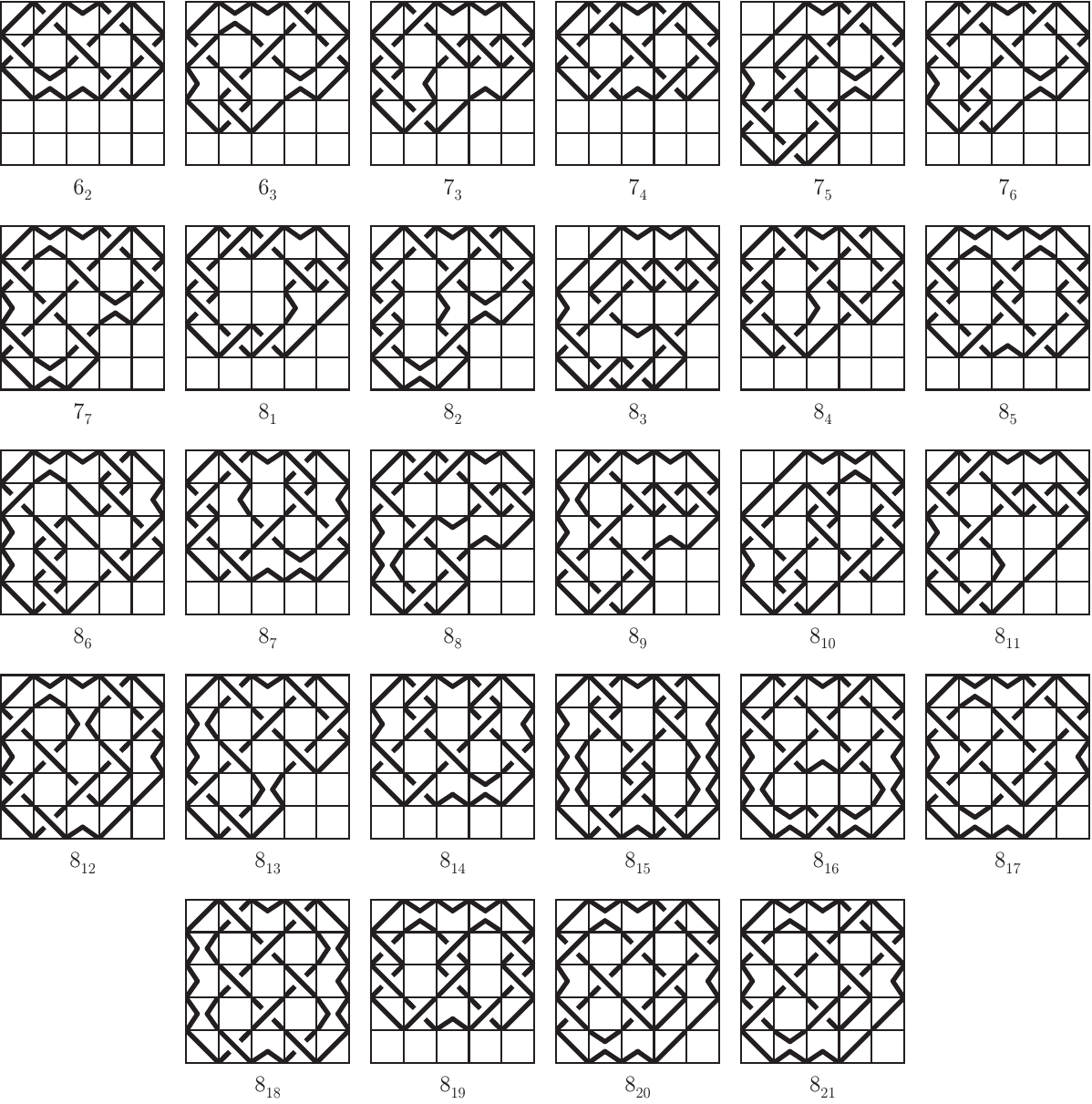}\\
  \caption{The remaining prime knots with crossing number 8 or less have corner connection mosaic number 5.}
  \label{fig:5mosaic-knots}
\end{figure}

\begin{theorem}\label{thm:5mosaic-knots}
  The knots $6_2$, $6_3$, $7_3$, $7_4$, $7_5$, $7_6$, $7_7$, and every prime knot with crossing number 8 have corner connection mosaic number 5. Of these, the known corner connection tile numbers are $\tcc(6_2)=\tcc(7_4)=13$ and $\tcc(6_3)=\tcc(7_6)=\tcc(8_1)=\tcc(8_4)=15$. For the remaining knots, $K$, in this set,  $15 \leq \tcc(K) \leq 20$.
\end{theorem}

\begin{proof}
  The mosaics in Figure \ref{fig:5mosaic-knots} verify the corner connection mosaic number is 5 for each of these knots. To get the corner connection tile numbers, we populate the layouts of Figure \ref{fig:5mosaic-layouts-14orless} with at least 6 crossing tiles. The only knots that come from the second layout are $6_2$ and $7_4$, and the only knots that come from the third layout are $6_1$ and  $7_2$, which were found previously. The next two layouts, with 14 non-blank tiles, are links when populated with 6 crossings. The final layout results in a link when an even number of crossing tiles are used, but with 7 or 9 alternating crossing tiles, the result is $7_1$ or $9_1$. All remaining 5-mosaic layouts have 15 or more non-blank tiles, and the mosaics in Figure \ref{fig:5mosaic-knots} have at most 20 non-blank tiles. By Theorem \ref{lemma:minimal-nonblank-tiles} we are able to make the given conclusions about the corner connection tile numbers.
\end{proof}

\begin{theorem}
  For all prime knots $K$ with crossing number 8 or less, $\mcc(K)\leq m(K)$ and $\tcc(K)\leq t(K)$.
\end{theorem}

\begin{proof}
  We noted in Section \ref{section:introduction} that using the corner connection tiles can produce more space-efficient mosaics for the unknot and the prime knots $3_1$, $4_1$, $5_1$, $6_1$, and $6_3$. Theorems \ref{thm:4mosaic-knots}, \ref{thm:tile-number-not-minimal}, and \ref{thm:5mosaic-knots} allow us to reach the same conclusion for every prime knot with crossing number 8 or less. The mosaic number of $5_2$, $6_2$, and $7_4$ is 5, and the remaining knots with crossing number 8 or less have mosaic number 6 \cite{Lee2}. In all cases, the corner connection mosaic number is less than or equal to the mosaic number. The tile number of $5_2$, $6_2$, and $7_4$ is 17, the remaining knots with crossing number 6 and 7 have tile number 22, as do some with crossing number 8 \cite{Heap1}. The remaining knots with crossing number 8 have tile number 24 or 27 \cite{Heap2}. Using corner connection tiles improves on all of these.
\end{proof}

\section{Maximum Number of Crossings}\label{section:max-crossing}

While using corner connection tiles allows us to create knot mosaics more efficiently for knots with small crossing number, this is not true for all knots. To show this, we first determine the maximum number of crossing tiles that can fit on an $n$-mosaic using the corner connection tiles.

\begin{theorem}\label{thm:max-cross}
  For any $n\geq3$, the upper bound for the number of crossing tiles used in an $n$-mosaic created from corner connection tiles is $n^2/2$ if $n$ is even and $(n^2+n-4)/2$ if $n$ is odd.
\end{theorem}

\begin{proof}
For a 3-mosaic, there are five non-corner tile locations that could contain a crossing. If the center tile position does not have a crossing, there are at most four crossing tiles in the mosaic. If the center tile position is a crossing, then by Lemma \ref{lemma:2by2}, there are at most three crossing tiles in the mosaic, with all of them being in the center row or all of them in the center column. 


Now consider an arbitrary $n$-mosaic, with $n \geq 4$. In the case where $n$ is even, the mosaic has $n^2/4$ non-overlapping $2 \times 2$ sub-mosaics, and by Lemma \ref{lemma:2by2}, each has at most two crossing tiles. Therefore, there are at most $n^2/2$ crossings in the entire mosaic. See the first mosaic in Figure \ref{fig:maximal-crossings} for an example.

\begin{figure}[ht]
  \centering
  \includegraphics[scale=.7]{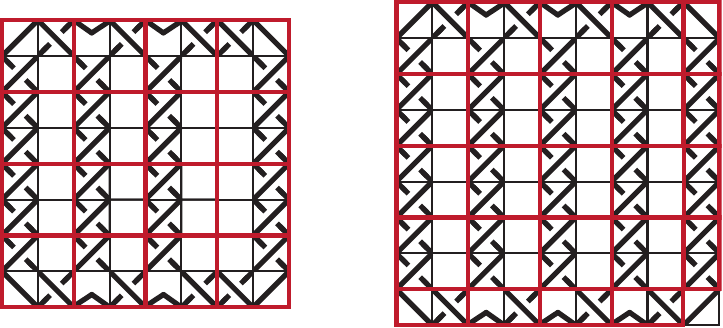}\\
  \caption{Even and odd examples of possible maximal crossing placements.}
  \label{fig:maximal-crossings}
\end{figure}

In the case where $n$ is odd, we note that $n-1$ is even, and looking at the original $n$-mosaic but ignoring the last row and column, the remaining $(n-1)\times (n-1)$ region has $(n-1)^2/4$ non-overlapping $2\times 2$ sub-mosaics, each with at most two crossings, for a maximum of $(n-1)^2/2$ crossings. Looking at the last column, we know the first and last tile cannot be a crossing, and there are $(n-1)/2$ non-overlapping $2 \times 1$ blocks. There is at most one crossing in the first $2 \times 1$ block. For the $(n-3)/2$ remaining $2 \times 1$ blocks, the number of possible crossings is restricted by Lemma \ref{lemma:2by2} and depends on the placement of the crossings in the adjacent $2\times 2$ sub-mosaics and on the other adjacent blocks.


Similarly for the last row, there is at most one crossing in the first $1 \times 2$ block, and there are $(n-3)/2$ remaining $1 \times 2$ blocks. This gives a total of $n-3$ non-overlapping $2 \times 1$ and $1 \times 2$ blocks in the last column and row. At most half of these have two crossings, and the other half have at most one crossing. See the second mosaic in Figure \ref{fig:maximal-crossings} for an example. Adding all of the above crossings together, we get the desired number of crossings:
$$\frac{(n-1)^2}{2} + 1 + 1 + 2\left(\frac{n-3}{2}\right) + \frac{n-3}{2} = \frac{n^2+n-4}{2}.$$
\end{proof}

We note that, besides the layouts given in Figure \ref{fig:maximal-crossings}, there are other layouts that realize the maximum number of crossings in an $n$-mosaic when $n$ is even. Some examples are given in Figure \ref{fig:maximal-layouts}. However, in the case where $n$ is odd the only way to obtain an $n$-mosaic with the maximum number of crossings is to use a layout that is structured similarly to the second layout given in Figure \ref{fig:maximal-crossings}. Deviations from this pattern result in fewer crossings because of Lemma \ref{lemma:2by2} or in kinks that reduce the number of crossings when the unnecessary crossings are resolved.

\begin{figure}[ht]
  \centering
  \includegraphics[width=\linewidth]{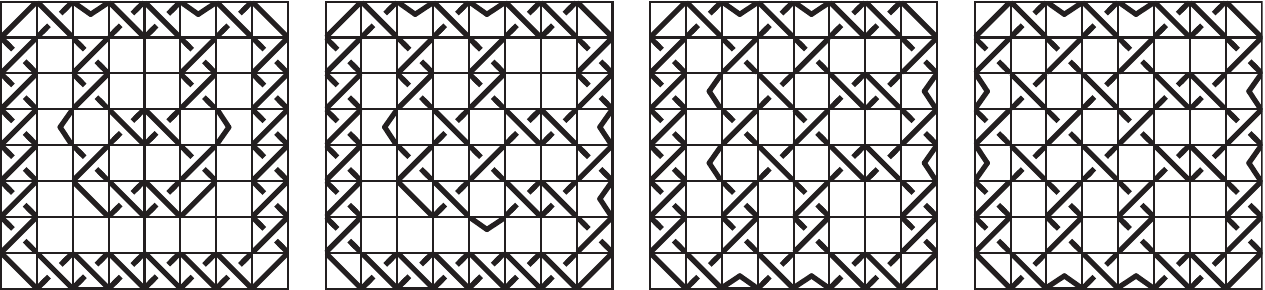}\\
  \caption{Examples of knot mosaics that realize the maximum number of crossings.}
  \label{fig:maximal-layouts}
\end{figure}

Although $\mcc(K)\leq m(K)$ for all prime knots $K$ with crossing number 8 or less, this is not the case for all prime knots. Now that we know the maximum number of crossing tiles in an $n$-mosaic when using the corner connection tiles, we can compare this to the maximum number of crossings allowed when using the traditional tiles. Howards and Kobin provided the latter in \cite{Howards}, showing that, for $n\geq 4$ the crossing number is bounded above by $(n-2)^2-(n-3)$ if $n$ is even and by $(n-2)^2-2$ if $n$ is odd. Comparing this to our bounds on the number of allowed crossings, $n=9$ is the first value for which the maximum number of crossings when using the corner connection tiles (43) is smaller than the maximum number of crossings when using traditional tiles (47), and this is where we find our counterexample.

\begin{theorem}\label{thm:counterexample}
  For sufficiently large knot mosaics, traditional tiles can lead to a smaller mosaic number than corner connection tiles do. That is, there exist knots $K$ such that $m(K) < \mcc(K)$.
\end{theorem}

\begin{proof}
As a result of Theorem \ref{thm:max-cross}, the maximum number of crossings on a 9-mosaic using corner connection tiles is 43. In Figure \ref{fig:counterexample}, we see a 9-mosaic using traditional tiles and 47 crossing tiles. By the Kauffman-Murasugi-Thistlethwaite Theorem \cite{Kauffman}, the crossing number of the depicted knot is 47, since it is a reduced, alternating knot diagram. Using corner connection tiles, this same knot can only be obtained on a 10-mosaic or larger.
\end{proof}

\begin{figure}[ht]
  \centering
  \includegraphics[scale=0.7]{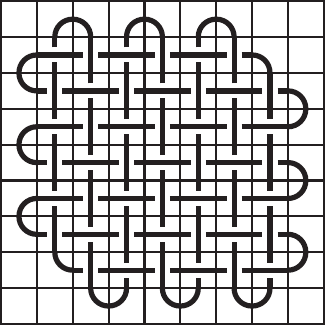} \\
  \caption{Counterexample showing that a knot made out of traditional tiles can fit on a smaller knot mosaic than one made with corner connection tiles.}
  \label{fig:counterexample}
\end{figure}

\section{Rectangular Mosaics}\label{section:rectangular}

There may be times when it is convenient to consider rectangular mosaics instead of square mosaics. An $(m,n)$-mosaic is a knot mosaic with $m$ rows and $n$ columns. For example, it may be more practical to use rectangular mosaics when working with pretzel links or braids. Moreover, the corner connection tiles allow for easier, often more efficient, depiction of the tangles.

\begin{figure}[ht]
  \centering
  \includegraphics{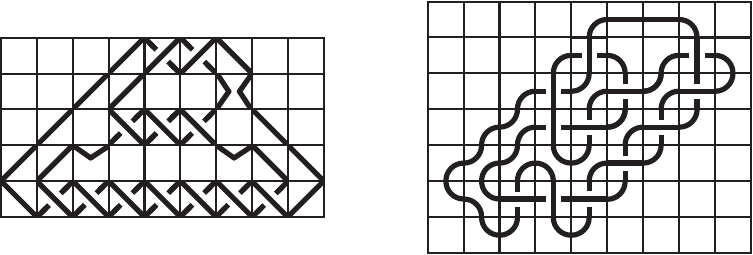}\\
  \caption{The Fintushel-Stern $(-2,3,7)$ pretzel knot on rectangular mosaics with corner connection tiles or with traditional tiles.}
  \label{fig:pretzel-knot}
\end{figure}

In Theorem \ref{thm:max-cross}, we gave bounds for the number of crossing tiles that can fit on a square mosaic using corner connection tiles. We can use a similar counting method to provide bounds for a rectangular mosaic.

\begin{theorem}
If $m\geq3$ and $n\geq3$, then the upper bound for the number of crossing tiles on an $(m,n)$-mosaic using corner connection tiles is:
\begin{enumerate}
  \item $mn/2$, if $m$ and $n$ are even;
  \item $(mn+n-4)/2$, if $m$ and $n$ are odd and $m\leq n$ or if $m$ is odd and $n$ is even.
\end{enumerate}
\end{theorem}

\begin{proof}
 The proof is almost identical to the proof of Theorem \ref{thm:max-cross}. In the case where $m$ and $n$ are both even, there are $mn/4$ $2 \times 2$ sub-mosaics, each with at most two crossings by Lemma \ref{lemma:2by2}, for a maximum of $mn/2$ crossings total.

 In the case when both $m$ and $n$ are odd, assuming $m \leq n$, there are $(m-1)(n-1)/4$ $2 \times 2$ sub-mosaics, each with at most two crossings. The last column has $(m-1)/2$ non-overlapping $2 \times 1$ blocks, and the last row has $(n-1)/2$ non-overlapping $1 \times 2$ blocks. The first of each of these can have at most one crossing. Since we are assuming $m \leq n$, we obtain the maximum number of crossings by assuming the remaining $(n-3)/2$ $1 \times 2$ blocks in the last row each have two crossings, forcing the remaining $(m-3)/2$ $2 \times 1$ blocks in the last column to have at most one crossing each.  There is also the bottom, right corner of the mosaic, which cannot be a crossing. Adding all of the above crossings together, we get the desired number of crossings:
$$\frac{(m-1)(n-1)}{2} + 1 + 1 + 2\left(\frac{n-3}{2}\right) + \frac{m-3}{2} = \frac{mn+n-4}{2}.$$

 In the case where $m$ is odd and $n$ is even, we ignore the last row, and the remaining $(m-1)\times n$ region has $(m-1)n/4$ non-overlapping $2\times 2$ sub-mosaics, each with at most two crossings. The last row has at most $n-2$ crossings since the two corner tiles cannot be crossings. Thus the total number of crossings is $(m-1)n/2 + n-2$, which simplifies to $(mn+n-4)/2$.
\end{proof}

\section{Further Questions}\label{section:questions}

There are plenty of open questions related to these corner connection tiles, and we provide a few of them here. The ones that come to our mind first are related to efficiency.

\begin{question}
  What is the simplest prime knot for which the mosaic number is smaller using traditional tiles than it is when using corner connection tiles?
\end{question}

\begin{question}
  Is it always true that $\tcc(K)\leq t(K)$? If not, what is the simplest prime knot for which the tile number is smaller using traditional tiles than it is when using corner connection tiles?
\end{question}

\begin{question}
  What is the corner connection mosaic number and corner connection tile number for each prime knot with 10 or fewer crossings?
\end{question}

Using traditional tiles it is known that there are infinitely many knots whose crossing number cannot be realized on a minimal mosaic \cite{Ludwig}, the simplest being $6_1$, $7_3$, $8_1$, $8_3$, $8_6$, $8_7$, $8_8$, and $8_9$. However, with corner connection tiles, there are no such knots with crossing number 8 or less.

\begin{question}
  Are there any knots where the crossing number cannot be realized when the corner connection mosaic number is?
\end{question}

\begin{question}
  Are there any knots where the crossing number cannot be realized when the corner connection tile number is?
\end{question}

Using traditional tiles, bounds for the tile number on an $n$-mosaic were determined in \cite{Heap1}, but we have yet to determine the bounds when using corner connection tiles.

\begin{question}
  What are the lower and upper bounds for corner connection tile number in terms of the size of the mosaic?
\end{question}

When using traditional tiles, there is a collection of sub-mosaic planar isotopy moves given in \cite{Lom-Kauff}, where it is claimed that this set of mosaic planar isotopy moves is complete. That is, every mosaic planar isotopy move is a composition of a finite sequence of the moves in the given collection.

\begin{question}
  What is the complete generating set of mosaic planar isotopy moves for corner connection tiles?
\end{question}

There are also some questions that remain unanswered for both the traditional tiles and the corner connection tiles. Most of the known mosaic numbers and tile numbers are for prime knots. It would be interesting to expand what is known for composite knots and links.

\begin{question}
  Using either set of square tiles, what can be said about the mosaic number and tile number for composite knots? In particular, for knots $K_1$ and $K_2$, is there a relationship between these mosaic invariants for $K_1$, $K_2$, and $K_1 \# K_2$?
\end{question}

\begin{question}
  Using either set of square tiles, what is the mosaic number and tile number for each multi-component link with 10 or fewer crossings?
\end{question}

Although the complete set of prime knots that can be placed on an $n$-mosaic, using traditional tiles, for each $n \leq 6$ has been determined, there is much more work to be done.

\begin{question}
  Using either set of square tiles, what is the complete set of prime knots, composite knots, and multi-component links can be placed on an $n$-mosaic for each $n \geq 3$?
\end{question}

%


\bibliographystyle{amsplain}
\bibliography{bibliography}
\addcontentsline{toc}{section}{\refname}

\end{document}